\documentclass{amsart}
\usepackage{amsmath,amsthm}
\usepackage{amsfonts,amssymb}

\usepackage{graphicx}
\usepackage{ifpdf}

\newtheorem{thm}{Theorem}[section]
\newtheorem{cor}[thm]{Corollary}

\newtheorem{prop}[thm]{Proposition}
\newtheorem{exam}[thm]{Example}

\theoremstyle{remark}

\def\f{\frac}

 \def\a{{\alpha}} 
 
 \def\g{{\gamma}}

 \def\d{{\delta}}

 \def\s{{\sigma}}

 \def\sb{{\mathbf s}}
 \def\tb{{\mathbf t}}

 \def\CC{{\mathbb C}}

 \def\NN{{\mathbb N}}

 \def\RR{{\mathbb R}}

 \def\sH{{\mathsf H}}
 
 \newcommand{\e}{\mathrm{e}}

        \def\sign{\operatorname{sign}}

\newcommand{\wh}{\widehat}

\begin{document}
 
\title[] 
{ Positivity and Fourier integrals over regular hexagon}
\author{Yuan Xu}
\address{Department of Mathematics\\ University of Oregon\\
 Eugene, Oregon 97403-1222.}\email{yuan@uoregon.edu}

\date{\today}
\thanks{The work was supported in part by NSF Grant DMS-1510296}

\begin{abstract}
Let $f \in L^1(\RR^2)$ and let $\wh f$ be its Fourier integral. We study summability of the partial integral 
$S_{\rho,\sH}(x)=\int_{\{\|y\|_\sH \le \rho\}} e^{i x\cdot y}\wh f(y) dy$, where $\|y\|_\sH$ denotes the uniform 
norm taken over the regular hexagonal domain.  We prove that the Riesz $(R,\delta)$ means of the inverse 
Fourier integrals are nonnegative if and if $\d \ge 2$. Moreover, we describe a class of $\|\cdot\|_\sH$-radial 
functions that are positive definite on $\RR^2$.   
\end{abstract}

\maketitle

\section{Introduction}
\setcounter{equation}{0}

The classical Bochner-Riesz means of the Fourier integral have kernels that are radial functions, or the 
$\|\cdot\|_2$-radial functions, where $\|\cdot\|_2$ denotes the usual Euclidean norm. We study their 
analogues that have kernels being $\|\cdot\|_\sH$-radial functions, where $\|\cdot\|_\sH$ denotes the uniform 
norm of the regular hexagonal domain of $\RR^2$, and $\|\cdot\|_\sH$-radial functions that are positive 
definite functions on $\RR^2$.

Let $f$ be a function in $L^1(\RR^d)$. The Foruier transform $\wh f$ and its inverse are defined by 
\begin{equation}\label{eq:oldFT}
\wh f(y) = \f1{(2\pi)^d} \int_{\RR^d} e^{-i x \cdot y} f(x) dx \quad\hbox{and}\quad  f(x) = \int_{\RR^d} e^{i y\cdot x} \wh f(y) dy, 
\end{equation}
where the latter integral need not exist for an arbitrary function $f \in L^1(\RR^d)$, a fact that motivates the study of
summability methods. The classical Bochner-Riesz means (cf. \cite{SW}) of the inverse Fourier transform are defined by  
\begin{equation} \label{eq:BR-means}
  S_{R, \delta}^{(2)}f(x) = \int_{ \|y\|_2 \le R } \left(1- \f{\|y\|_2}{R} \right)^\d e^{i y\cdot x} \wh f(y) dy.
\end{equation}
The convergence of these means has been studied extensively. If $\|y\|_2$ is
replaced by the $\ell_1$ norm $|y|_1: =|y_1|+\ldots+|y_d|$ in \eqref{eq:BR-means}, we denote the new means by 
$S_{R,\delta}^{(1)} f$ and call them $\ell_1$-Riesz $(R,\delta)$ means. It was proved in \cite{BX} that, in $\ell_1$ 
summability, the $(R,\delta)$ means $S_{R,\delta}^{(1)} f$ define positive linear transformations on $L^1(\RR^d)$ 
exactly when $\delta \ge 2d-1$. In contrast, in $\ell_2$ summability, the Bochner-Riesz means do not define positive 
transformations for any $\d >0$ \cite{Go}.

In the present paper we study the case when $\|\cdot\|_2$ in \eqref{eq:BR-means} is replaced by the uniform norm 
$\|\cdot\|_\sH$ over the regular hexagonal domain in $\RR^2$. In this case it is more convenient to work in homogeneous 
coordinates of 
$$
\RR_{\sH}^3: = \{\tb = (t_1,t_2,t_3) \in \RR^3: t_1+t_2+t_3 =0\},  
$$
for which the regular hexagonal domain is equivalent to $\{\tb \in \RR^3: \|\tb\|_\sH \le 1\}$, where 
$$
\|\tb\|_\sH:= \max_{1\le i \le 3} |t_i|.
$$ 
In $\RR_{\sH}^3$ the Fourier transform and its inverse can be defined by 
\begin{equation}\label{eq:FT}
\wh f(\sb) = \f1{3\pi^2} \int_{\RR_{\sH}^3} e^{- \frac{2i}{3}  \tb \cdot \sb} f(\tb) d\tb \quad\hbox{and}\quad  f(\tb) = \int_{\RR_{\sH}^3} 
   e^{\frac{2i}{3}  \tb \cdot \sb} \wh f(\sb) d\sb, 
\end{equation}
as we shall see in the next section. The Riesz  $(R,\delta)$ means are then become 
\begin{equation} \label{eq:R-H-means}
  S_{R, \delta} f(\tb) := \int_{ \|\sb\|_\sH \le R} \left(1- \f{\|\sb\|_\sH}{R} \right)^\d e^{i \sb\cdot \tb} \wh f(\sb) d\sb. 
\end{equation}

The symmetry of the regular hexagonal domain makes it possible to derive a close form for the Dirichlet kernel, 
$$
   D_R(\tb) = \int_{\|\tb\|_{\sH} \le R} e^{ \frac{2 i}{3} \sb \cdot \tb} d\sb, \qquad \tb \in \RR_{\sH}^3,
$$
which can be used to establish the following theorem. 

\begin{thm}\label{thm1}
For $\delta \ge 2$ the Riesz $(R,\delta)$ means of the hexagonal partial integral \eqref{eq:R-H-means} define positive linear 
transformations on $L^1(\RR_{\sH}^3)$; the order of the summability to assure positivity is best possible.
\end{thm}
 
We note that the minimal order of the summability to assure positivity of the Riesz $(R,\delta)$ means for $\ell_1$ 
summability is  $\delta \ge 3$ when $d=2$. 

A function $\phi: \RR_{\sH}^3 \mapsto \RR$ is called $\|\cdot\|_\sH$ invariant, or $\|\cdot\|_\sH$-radial, 
if $\phi(\tb) = \phi_0(\|\tb\|_\sH)$ for some $\phi_0: \RR_+ := [0,\infty) \mapsto \RR$. Although the 
Dirichlet kernel is not $\|\cdot\|_\sH$ radial, it has additional structure that allows us to characterize positive 
definiteness of such functions.

A function $\phi$ on $\RR^d$ is said to be {\it positive definite} if for any set of points $x_1,\ldots, x_N$ in $\RR^d$ and scalars $c_1,\ldots, c_N$ 
in $\CC$, $N \in \NN$, 
$$
   \sum_{j=1}^N  \sum_{k=1}^N c_j c_k \phi (x_j - x_k) \ge 0; 
$$
that is, the matrix $[\phi (x_j - x_k)]_{j, k =1}^N$ is positive semi-definite for all $\{x_j\}$, $\{c_j\}$ and $N$. Bochner 
proved in \cite{Bo} that a continuous function $\phi$ on $\RR^d$ is positive definite exactly when it is the Fourier 
integral of a finite positive measure on $\RR^d$. Shoenberg specialized Bochner's theorem for $\ell_2$ radial functions 
and proved that $\phi(\|x\|_2)$ is positive define exactly when 
\begin{equation} \label{eq:radial-pdf}
   \phi(t) = \int_0^\infty \Omega_d(t u) d\alpha(u) \quad \hbox{with} \quad 
       \Omega_d(r): = \left( \f2{r} \right)^{\f{d-2}{2}} J_{\f{d-2}{2}}(r)
\end{equation}
and $\a(u)$ is non-decreasing and bounded for $u \ge 0$, where $J_\g$ is the Bessel function. An analogue result was 
proved in \cite{BX} for $\ell_1$-radial functions $\phi(|x|_1)$, for which $\phi$ is characterized by \eqref{eq:radial-pdf} 
with $\Omega_d$ replaced by $m_d$, where $m_d$ is a function that can be defined either recursively or as the Fourier
transform of the B-spline function $x\mapsto M(\cdot| x_1^2,\ldots,x_d^2)$ on $\RR^d$. In particular, $m_2$ is given by 
$$
   m_2(\xi) =   \int_\xi^\infty \frac{\sin u}{u} du.
$$
In the case of $d =2$, under the change of variables $x = s+t$ and $y=s-t$, the $\ell_1$ ball $\{x\in \RR^2: |x|_1 \le 1\}$ 
becomes the square $[-1,1]^2$  and the $\ell_1$-radial functions become $\|\cdot\|_\infty$, or $\ell_\infty$-radial functions. Hence, the result in \cite{BX} also gives a characterization of $\ell_\infty$ radial functions. 

Our second  result shows that the class of positive definite $\|\cdot\|_\sH$ radial functions is at least as big as 
the class of positive definite $\ell_1$ or $\ell_\infty$ radial functions. 

\begin{thm}\label{thm2}
Let $\phi \in C_b(\RR_+)$. The function $\phi(\|\tb\|_\sH)$, $\tb \in \RR_{\sH}^3$, is positive definite on $\RR_{\sH}^3$ 
if there exists an increasing bound function $\a$ on $\RR_+$ such that 
\begin{equation*}
 \phi(t) = \int_0^\infty  m_2(t u)  d \a(u). 
\end{equation*}
\end{thm}

We do not know if the converse of Theorem \ref{thm2} holds; that is, whether $\phi(\|\tb\|_\sH)$ is positive definite {\it only if}
$\phi$ is of the form given in the theorem. In \cite{BX}, the inverse in the $\ell_1$ case was established by writing the 
Dirichlet kernel as an integral against a B-spline function whose knots are the variables, and studying the Fourier 
transform of the B-spline function. The kernel in the hexagonal setting can also be written as an integral against a 
B-spline function, but the resulted B-spline function is no longer integrable, which prevents us from following the 
same approach.   

The paper is organized as follows. The set-up of the Fourier integral on the hexagonal domain and the analysis, 
based on the closed formula for the Dirichlet kernel, that leads to the proof of Theorem \ref{thm1} are given in the next
section. The positive definiteness of $ \|\cdot\|_\sH$-radial functions is discussed in Section 3. 
 
\section{Fourier integral on the hexagonal domain}
\setcounter{equation}{0}
 
The regular hexagonal domain on $\RR^2$ is given by 
$$
   H =\left\{(x_1,x_2):\  -1\leq x_2, \tfrac{\sqrt{3}}{2}x_1 \pm 
   \tfrac{1}{2} x_2 < 1 \right\}.
$$ 
We consider the Fourier integral $\int_{H} \wh f(y) e^{- i x\cdot y} dy$. Such integrals have been studied, say, in 
\cite{AB}, where the Tur\'an's problem on positive definite functions on the hexagonal domain is studied.
For our purpose, it is more convenient to use homogeneous coordinates $\tb = (t_1,t_2,t_3)$ in the space 
$$
    \RR_{\sH}^3 : = \{\tb = (t_1,t_2,t_3)\in \RR^3: t_1+t_2 +t_3 =0\}, 
$$
for which the hexagonal domain $H$ becomes 
\begin{align*}
   \sH:=\left\{\tb \in \RR_{\sH}^3:\  -1\le  t_1,t_2, t_3 \le 1 \right\} = \{\tb \in \RR_{\sH}^3: \|\tb\|_\sH \le 1\}, 
\end{align*}
where $\|\tb\|_\sH= \max_{1\le i \le 3} |t_i|$. Geometrically $\sH$ is the intersection of the plane $t_1+t_2+t_3=0$ 
with the cube $[-1,1]^3$ as shown in Figure 1. 
Such a convenience is used in \cite{LSX,X10} where the Fourier series associated with the hexagonal lattice and
their applications in discrete Fourier transform are studied.
\begin{figure}[ht]
\centerline{
\includegraphics[width=0.375\textwidth]{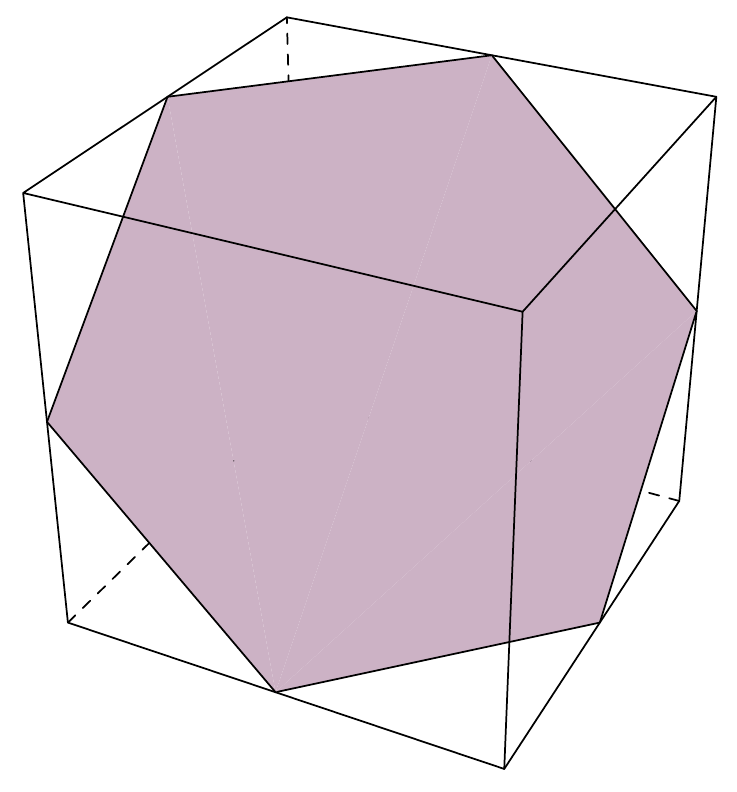}}
\caption{Regular hexagon in $\RR_{\sH}^3$.}
\end{figure}

For convenience, we adopt the convention of using bold letters, such as $\tb$, to denote points in homogeneous
coordinates of $\RR_{\sH}^3$ throughout this paper. The relation between $(x_1,x_2) \in H$ and $\tb \in \sH$ is
given by 
\begin{align*}
 t_1= -\frac{x_2}{2} +\frac{\sqrt{3}x_1}{2},\quad t_2 = x_2, \quad t_3 = 
      -\frac{x_2}{2} -\frac{\sqrt{3}x_1}{2}, 
\end{align*}
or, equivalently, using the fact that $t_1+t_2+t_3 =0$, 
\begin{equation} \label{x-t-x}
x_1 =  \tfrac13 (t_1-t_3) =   \tfrac13 (2 t_1+t_2), \quad x_3 =  \tfrac13 (t_2-t_3) =   \tfrac13 ( t_1+2t_2).
\end{equation}
Let $d \tb$ be the Lebesgue measure on $\RR_{\sH}^3$. Computing the Jacobian of the change of variables shows 
that $d x_1 dx_2 = \frac{2 \sqrt{3}} {3} d \tb$. With $x$ and $y$ associated to $\tb$ and $\sb$, respectively,
through \eqref{x-t-x}, it is easy to see that $x\cdot y = \frac23 \sb\cdot \tb$. If we identity the function $f(x)$ on $\RR^2$ with 
$f(\tb)$ on $\RR_{\sH}^3$, then the Fourier transform and its inversion \eqref{eq:oldFT} translate to   
\begin{equation*} 
  \wh f(x) \mapsto \frac{\sqrt{3}}{ 6 \pi} \int_{\RR_H^3} f(\tb) \e^{- \frac{2 i}{3} \sb \cdot \tb} d\tb
   \quad \hbox{and} \quad
  f(x) \mapsto \frac{2\sqrt{3}}{3} \int_{\RR_H^3} \wh f(\sb) \e^{\frac{2 i}{3} \sb \cdot \tb} d\sb. 
\end{equation*} 
For convenience we renormalize them and defined the Fourier transform and its inverse on $\RR_{\sH}^3$ as in \eqref{eq:FT},
that is,
\begin{equation*}
\wh f(\sb) = \f1{3\pi^2} \int_{\RR_{\sH}^3} e^{- \frac{2i}{3}  \tb \cdot \sb} f(\tb) d\tb \quad\hbox{and}\quad  f(\tb) = \int_{\RR_{\sH}^3} 
   e^{\frac{2i}{3}  \tb \cdot \sb} \wh f(\sb) d\sb. 
\end{equation*} 
The usual definition for the convolution $f*g$ extends to $f,g\in L^1(\RR_{\sH}^3)$ by
$$ 
   f*g(\tb) = \int_{\RR_\sH^3} f(\tb -\sb) g(\sb) d\sb, \qquad f,g \in L^1(\RR_\sH^3).
$$  

For $\rho > 0$ we first consider the hexagonal summability of the inverse Fourier integral 
\begin{equation}\label{eq:Srho}
    \sigma_\rho (f; \tb) = \int_{\|\tb\|_\sH \le \rho} \wh f(\sb)  e^{\frac{2i}{3}\sb \cdot \tb} d\sb  = (f \ast D_\rho)(\tb), 
\end{equation}
where $D_\rho$ is the Dirichlet kernel for the regular hexagonal domain, 
$$
  D_\rho (\tb) : =  \int_{\|\tb\|_\sH \le \rho}   e^{- \frac{2i}{3}\sb \cdot \tb} d\sb, \qquad \tb \in \RR_\sH^3.
$$
Our first result is a closed from for the Dirichlet kernel for the hexagonal domain.

\begin{prop} \label{prop:D-kernel}
For $\rho > 0$, 
\begin{equation}\label{eq:Drho}
    D_\rho(\tb) = -  \frac{9}{2} \left[ \frac{ \cos \left[\frac{2 }{3} \rho(t_1-t_2)\right]} {(t_2-t_3)(t_3-t_1)} +
              \frac{ \cos \left[\frac{2 }{3} \rho(t_2-t_3)\right]} {(t_3-t_1)(t_1-t_2)} +  
              \frac{ \cos \left[\frac{2 }{3} \rho(t_3-t_1)\right]} {(t_1-t_2)(t_2-t_3)}\right].
\end{equation}  
\end{prop}

\begin{proof}
Let $D(\tb):= D_1(\tb)$. A simple change of variable shows that 
$$
    D_\rho (\tb) =  \rho^2 D (\rho \tb). 
$$ 
Hence, we only need to work with the case $\rho =1$. The hexagonal domain can 
be partitioned into three parallelograms, as shown in Figure 2, which leads to 
\begin{figure}[ht]
\centerline{
\includegraphics[width=0.5\textwidth]{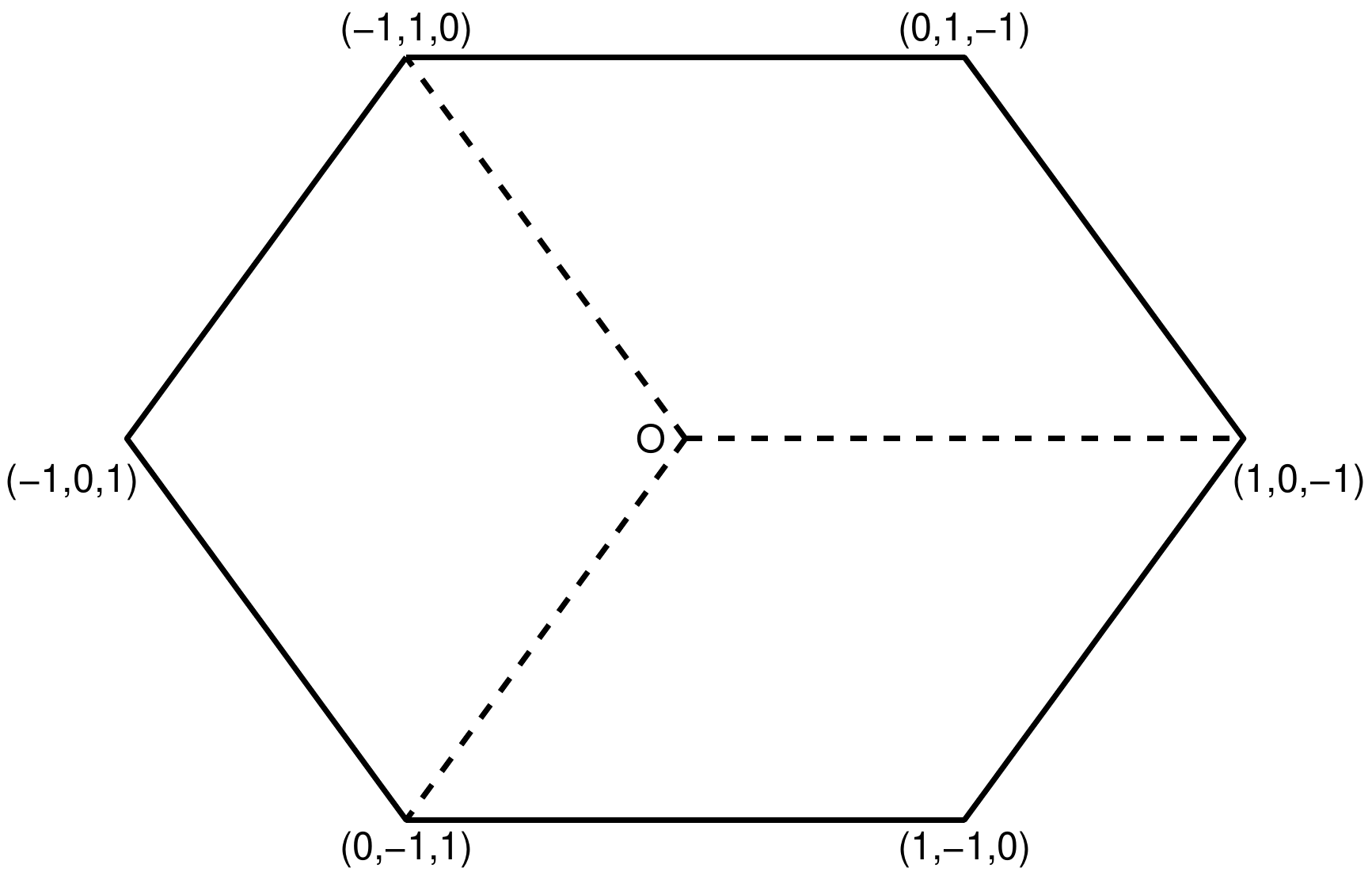}}
\caption{Hexagon and its partition in $\RR_{\sH}^3$.}
\end{figure}
\begin{align*}
  D(\tb) = \int_0^1 \int_{-1}^0 e^{\frac{2 i}{3}\sb \cdot \tb} ds_1 ds_2
         + \int_0^1 \int_{-1}^0 e^{\frac{2 i}{3}\sb \cdot \tb} ds_2 ds_3 +
          \int_0^1 \int_{-1}^0 e^{\frac{2 i}{3}\sb \cdot \tb} ds_3 ds_1.
\end{align*}
For $\sb, \tb \in \RR_H^3$, we can write $\sb \cdot \tb = (t_1-t_3)s_1 + (t_2-t_3)s_2$ 
by using $s_1+s_2+s_3 =0$, so that the first integral can be easily evaluated as 
$$
 I(1,2):= \int_0^1 \int_{-1}^0 e^{\frac{2 i}{3}\sb \cdot \tb} ds_1 ds_2 
    =  - \left(\frac{3}{2} \right)^2 \frac{\left(1- e^{\frac{2 i}{3}(t_1-t_3)}\right)
          \left (1- e^{- \frac{2 i}{3}(t_2-t_3)}\right) }{(t_1-t_3)(t_2-t_3)}. 
$$
The other two integrals are evidently just $I(2,3)$ and $I(3,1)$. The nominator of the last fraction in the right hand side of 
$I(1,2)$ can be written as 
$$
1- e^{-\frac{2 i}{3}(t_3-t_1)}- e^{-\frac{2 i}{3}(t_3-t_1)}- e^{-\frac{2 i}{3}(t_3-t_1)} + 2  \cos \left[\frac{2 }{3}(t_1-t_2) \right]. 
$$
Without the cosine term, the above expression is invariant under the permutation. Hence, using the fact that 
$(t_1-t_2) + (t_2-t_3)+(t_3-t_1) =0$, summing up $I(1,2) + I(2,3)+I(3,1)$ shows that 
\begin{equation*}
    D(\tb) = -  \frac{9}{2} \left[ \frac{ \cos \left[\frac{2 }{3} (t_1-t_2)\right]} {(t_2-t_3)(t_3-t_1)} +
              \frac{ \cos \left[\frac{2 }{3}  (t_2-t_3)\right]} {(t_3-t_1)(t_1-t_2)} +  
              \frac{ \cos \left[\frac{2 }{3}  (t_3-t_1)\right]} {(t_1-t_2)(t_2-t_3)}\right],
\end{equation*}  
which is the desired formula for $\rho =1$. This completes the proof. 
\end{proof}  

Along the same line of the proof, we have the following proposition on the Fourier transform of $\|\cdot\|_\sH$ radial  
functions. 

\begin{prop} \label{prop:FT2}
Let $\phi \in C_0(\RR_+)$. For any $r > 0$ and $\tb \in \RR_H^3$, 
\begin{equation} \label{eq:d-Drho}
 \int_{\|\sb\|_\sH \le r} e^{\pm \frac{2 i}{3} \sb \cdot \tb } \phi (\|\sb\|_\sH) d \sb = \int_0^r E_\rho (\tb) \phi(\rho) d\rho 
\end{equation}
for almost all $\tb$, where 
$$
   E_\rho(\tb) = \frac{d}{d\rho} D_\rho(\tb). 
$$
\end{prop}

\begin{proof}
As in the proof of Proposition \ref{prop:D-kernel}, we can split the integral into three pieces, 
$$
\int_{\|\sb\|_\sH \le r} e^{ \frac{2 i}{3} \sb \cdot \tb } \phi (\|\sb\|_\sH) d \sb = J(1,2) + J(2,3) + J(3,1),
$$
where, writing $\sb \cdot t = s_1 (t_1-t_3)+s_2(t_2-t_3)$ for $\tb \in \RR_\sH^3$, 
\begin{align*}
  J(1,2) := & \int_0^r \int_{-r}^0 e^{ \frac{2  i}{3} \sb \cdot \tb } \phi (\|\sb\|_\sH) d s_1 ds_2 \\
      =  & \ \int_0^r \int_0^{r} e^{ \frac{2  i}{3}  (s_1 (t_1-t_3)-s_2(t_2-t_3))} \phi (\|\sb\|_\sH) d s_1 ds_2,
\end{align*}
and $J(2,3)$ and $J(3,1)$ are permutations of $J(1,2)$. For $s_1, s_2 > 0$, $\|\sb\|_\sH = \max\{s_1,s_2\}$. Hence, it follows that
\begin{align*}
 J(1,2)  =  &\int_0^r  e^{ \frac{2  i}{3} s_1 (t_1-t_3) } \phi (s_1) d s_1 \int_0^{s_1}  e^{ - \frac{2  i}{3} s_2 (t_2-t_3) } ds_2 \\
    & + \int_0^r  e^{ - \frac{2  i}{3} s_2 (t_2-t_3) } \phi (s_2) d s_2 \int_0^{s_2}  e^{\frac{2  i}{3} s_1 (t_1-t_3) } ds_1 \\
     = &\frac{3}{2}\int_0^r \phi(\rho) \left[ \frac{e^{ \frac{2  i}{3} \rho (t_1-t_3)} - e^{\frac{2  i}{3} \rho (t_1-t_2)}}{i (t_2-t_3)} -
           \frac{e^{ \frac{2  i}{3} \rho (t_3-t_2)} -e^{ \frac{2  i}{3} \rho (t_1-t_2)}}{i (t_1-t_3)} \right] d\rho. 
\end{align*}
The terms inside the bracket can be rewritten as 
$$
  \frac{(t_1-t_3) e^{\frac{2  i}{3} \rho (t_1-t_3)}+(t_3-t_2) e^{\frac{2  i}{3} \rho (t_3-t_2)}
        - (t_1-t_2) e^{ \frac{2  i}{3} \rho (t_1-t_2)}}{i (t_2-t_3)(t_1-t_3)} 
$$
and the nominator of this expression can be further rewritten as 
\begin{align*}
(t_1-t_3) e^{\frac{2  i}{3} \rho (t_1-t_3)}+(t_3-t_2) e^{\frac{2  i}{3} \rho (t_3-t_2)}
       + (t_2-t_1) e^{ \frac{2  i}{3} \rho (t_2-t_1)} \\
         -  2i (t_1-t_2) \sin \left[\frac{2 }{3} \rho (t_1-t_2)\right],
\end{align*}
in which the sum in the first line is invariant under permutation. Consequently, adding $J(1,2)$, $J(2,3)$ and $J(3,1)$ gives
\begin{align*}
  \int_{\|\tb\|_\sH \le r} e^{ \frac{2  i}{3} \sb \cdot \tb } \phi (\|\sb\|_\sH) d \sb 
   =     \int_0^r \phi(\rho)  E_\rho(\tb)  d\rho
\end{align*}
where 
\begin{align}\label{eq:Erho}
 E_\rho(\tb) = & \ 3 \left[ \frac{(t_1-t_2) \sin \left[\frac{2} {3} \rho (t_1-t_2)\right]}{(t_2-t_3)(t_3-t_1)} \right. \\
       &  + \left. \frac{(t_2-t_3) \sin \left[\frac{2}{3} \rho (t_2-t_3)\right]}{(t_3-t_1)(t_1-t_2)}  
          +      \frac{(t_3-t_1) \sin \left[\frac{2}{3} \rho (t_3-t_1)\right]}{(t_1-t_2)(t_2-t_3)} \right]. \notag
\end{align}
It is now easy to see that $\frac{d}{d\rho} D_{\rho}(\tb) =  E_\rho(\tb)$ and \eqref{eq:d-Drho} follows. 
\end{proof}

The computation of the proof also shows that 
\begin{equation*} 
  \int_0^r \phi(\|\tb\|_\sH) d\tb = 6 \int_0^r \rho \phi(\rho)d\rho. 
\end{equation*}

This proposition allows us to compute the Fourier transform of $\|\cdot \|_\sH$ radial functions. 
 
\begin{exam}\label{ex:2.3}
For $a > 0$ and $\tb \in \RR_\sH^3$, let $\phi(\tb) = e^{- \frac{2 a}{3} \|\tb\|_H}$. Then
\begin{align*}
  \wh \phi (\tb)   = &  \frac{ 27 a^2( 2a^2 + t_1^2+t_2^2+t_3^2)} 
       {4   (a^2+ (t_1-t_2)^2)(a^2+ (t_2-t_3)^2)(a^2+ (t_3-t_1)^2)}.
\end{align*}
\end{exam}

\begin{proof}
We use the explicit formula of $E_\rho$ in \eqref{eq:Erho} and the elementary integral
$$
  \int_0^\infty e^{-a \rho} \sin (b \rho) d\rho = \frac{b}{a^2+b^2}, \qquad a > 0,
$$
then simplify the computation using $t_1 t_2+t_2t_3 + t_3 t_1 = - (t_1^2+t_2^2+t_3^2)/2$, which 
comes from $(t_1+t_2+t_3)^2 =0$. 
\end{proof}

For $\delta > 0$, the Ces\`are $(C,\delta)$ means of a function $s: \RR_+ \mapsto \CC$ are defined by
$$
  s^\d(\rho) =  \frac{\d}{\rho} \int_0^\rho (\rho - u)^{\d-1} s(u) du, \qquad \rho > 0.
$$
Because of the the convolution structure of the partial integral \eqref{eq:Srho}, its $(C,\delta)$ means
can be viewed as convolving the function $f$ with the $(C,\delta)$ means of the Dirichlet kernel; that is, 
define 
$$
  D_R^\d (\tb):= \frac{\d}{R^\d} \int_0^R (R - \rho)^{\d -1} D_\rho(\tb) d\rho,
$$
then the $(C,\delta)$ means of the integral in \eqref{eq:Srho} is defined by 
$$
   \s_R^\d (f;\tb) = (f \ast D_R^\d)(\tb), \qquad  R > 0, \quad \tb \in \RR_H^3.
$$
Our next result shows that the Ces\`aro $(C,\delta)$ means and the Reize $(R,\delta)$ means, defined in
\eqref{eq:R-H-means}, of hexagonal summability of the Fourier integral are identical. 

\begin{cor} \label{cor:C=R}
Let $f \in L^1(\RR_H^2)$ and $\delta > 0$. Then for any $r>0$, 
\begin{align*}
   S_{R,\delta}f(\tb) = \ & \int_{\|\sb\|_\sH \le 1} \left(1-\frac{\|\sb\|_\sH}{R}\right)^\delta e^{ \frac{i}{3} \sb\cdot\tb} \wh f(\sb)d\sb \\
       =  \ & \frac{\d}{R^\d} \int_0^R (R -\rho)^{\d-1} \s_\rho (f; \tb) d\rho = (f*D_R^\d)(\tb), \qquad \tb \in \RR_H^3.
\end{align*}
\end{cor}
 
\begin{proof}
Let $\phi \in C_0(\RR_+)$ be locally absolutely continuous. Integration by parts in the right hand side of \eqref{eq:d-Drho} shows
that 
$$
 \int_{\|\sb\|_\sH \le R} \phi(\|\sb\|_\sH) e^{ \frac{i}{3} \sb\cdot\tb} d\sb = \phi(R)D_{R}(\tb) - \int_0^R  \phi'(\rho) D_\rho(\tb) d\rho.    
$$
Setting $\phi(t) = (1-t/R)_+^\d$ for $t > 0$ and $\d > 0$, this identity becomes
$$
 \int_{\|\sb\|_\sH \le R} \left(1 - \frac{\|\sb\|_\sH}{R} \right)^\d  e^{ \frac{i}{3} \sb\cdot\tb} d\sb
    = \frac{\d}{R^\d} \int_0^R (R-\rho)^{\d-1} D_\rho(\tb) d\rho. 
$$ 
Taking the convolution with $f$ proves the corollary. 
\end{proof}
 
By the close form of the Diriclet kernel in \eqref{eq:Drho}, we immediately conclude that 
\begin{equation} \label{C-kernel}
    D_R^\d (\tb) = - \frac{9}{2} \left[ \frac{ F_\d \left(\frac23 R (t_1-t_2)\right) } {(t_2-t_3)(t_3-t_1)} +
      \frac{F_\d \left(\frac23 R (t_2-t_3)\right) } {(t_3-t_1)(t_1-t_2)} +  \frac{ F_\d \left(\frac23 R (t_3-t_1)\right) } {(t_1-t_2)(t_2-t_3)}\right],
\end{equation}
where 
$$
   F_\d (u): = \d \int_0^1 \cos \left ( \rho u\right) (1-\rho)^{\d-1} d \rho, \quad u > 0. 
$$
Elementary computation shows that $F_\d$ can be written as a ${}_1F_2$ series
$$
   F_\d(t) = {}_1F_2(1; \tfrac{\d+1}{2}, \tfrac{\d+2}{2}; - \tfrac{t^2}{4}).
$$
For some special values of $\d$, $F_\d$ enjoys compact expression. For example, 
\begin{equation} \label{F2} 
   F_1(t) = \frac{\sin t}{t} \quad \hbox{and} \quad F_2(t) = \frac{2(1- \cos t)}{t^2}.
\end{equation}
The kernel $D_R^\d$ turns out to be positive for $\d \ge 2$. More precisely, we prove the following theorem. 
\begin{thm} 
The kernel $D_R^\d(\tb)$ is nonnegative on $\RR^3_H$ if, and only if, $\delta \ge 2$. 
\end{thm} 

\begin{proof}
First we prove that $D_R^\d(\tb) \ge 0$ if $\delta \ge 2$. For $\d, \mu >0$ it is easy to verify that 
$$
D^{\d+\mu}_R (\tb) = \frac{\Gamma(\d+\mu+1)}{\Gamma(\d+1)\Gamma(\mu)} 
    \frac{1}{R^{\d+\mu}} \int_0^R (R - \rho)^{\mu-1} \rho^\d \,D_\rho^\d (\tb) d\rho,
\qquad  R > 0.
$$
Thus, it follows that $D^{\d+\mu}_R(\tb)$ is nonnegative if $D^{\d}_R(\tb)$ is for all $R$. Hence, it suffices to show 
that $D_R^2(\tb)$ is nonnegative for $\tb \in \RR_\sH^3$ and $R>0$.  Using the close form \eqref{C-kernel} of $D_R^\delta$
and the explicit formula of $F_2$ in \eqref{F2}, it follows readily that $D_R^2(\tb) \ge 0$ if 
\begin{align*}
    G_R(\tb) : =& - (t_2-t_3)(t_3-t_1)\left(1-\cos [\tfrac{2}{3}R(t_1-t_2)]\right) \\
      &   - (t_3-t_1)(t_1-t_2)\left(1-\cos [\tfrac{2}{3}R(t_2-t_3)]\right) \\
      &   - (t_1-t_2)(t_2-t_3)\left(1-\cos [\tfrac{2}{3}R(t_3-t_1)]\right)
\end{align*} 
is nonnegative. Evidently, it suffices to establish the non-negativity of $G_R(\tb)$ when $ \frac{2 }{3} R = 1$, which we 
denote by $G(\tb)$. It turns out that $G(\tb)$ can be written as a sum of square, from which the nonnegativity of $G(\tb)$ 
follows immediately. Indeed, the following identity holds, 
\begin{align*}
   G(\tb) = &  \frac{1}{2} \left[ (t_1 - t_2) \cos t_3 + (t_2 - t_3) \cos t_1 + (t_3 - t_1) 
     \cos t_2 \right]^2 \\
      & + \frac{1}{2} \left[ (t_1 - t_2) \sin t_3 + (t_2 - t_3) \sin t_1 + (t_3 - t_1) 
        \sin t_2\right]^2. 
\end{align*}
The difficult lies in identifying the formula. The verification is tedious but straightforward, and it can be checked by
a computer algebra system. This proves the positivity of $D_R^2(\tb)$. 

Next we prove that $D_R^\d(\tb)$ is not nonnegative when $0 < \d < 2$. We only need to consider the case
$1< \delta < 2$, since if $D_R^\d(\tb)$ is nonnegative for some $\delta$ that satisfies $0 < \delta \le 1$, then it has to be 
nonnegative for $1 < \d < 2$. Assume $1 < \d <2$. It suffices to show that $D_R^\d(\tb)$ 
is negative for some $\tb \in \RR^3_H$.  Using the explicit formula of
the \eqref{C-kernel}, it is easy to see that
$$
   D_R^\d \left(\frac{3 \pi}{R}, - \frac{3\pi}{R}, 0 \right) = \frac{ R^2}{2 \pi^2} 
        \left[F_\d(2 \pi) - F_\d (4 \pi) \right]. 
$$
Integrating by parts shows that 
\begin{align*}
   F_\d(2 \pi ) - F_\d(4 \pi) & =  \frac{\d(\d-1)}{4 \pi} \int_0^1 (1-s)^{\d -2} 
            [2 \sin (2\pi s)- \sin (4 \pi s) ] ds \\
         & =  \frac{\d(\d-1)}{2 \pi} \int_0^1 (1-s)^{\d -2} 
               \sin (2\pi s) [1-  \cos (2\pi s) ]  ds.
\end{align*}
Splitting the last integral as two, one over $[0,1/2]$ and the other over $[1/2,1]$,
and changing variable $t \mapsto 1-s$ in the second integral, we see that  
$$
  F_\d(2 \pi ) - F_\d(4 \pi)   =  \frac{\d(\d-1)}{2 \pi}
    \int_0^{1/2}  \left[ (1-s)^{\d -2} -  s^{\d-2} \right]\sin (2\pi s)[1-  \cos (2\pi s) ] ds. 
$$
Since for $0<s < 1/2$, $\sin (2\pi s)[1-  \cos (2\pi s) ]\ge 0$ and $(1-s)^{\d -2} 
-  s^{\d-2} < 0$ as $\d -2 < 0$, we conclude that $F_\d(2 \pi) - F_\d(4 \pi) < 0$ 
for $1 < \d < 2$. Consequently $D_R^\d \left(\frac{3\pi}{R}, - \frac{3 \pi}{R}, 0 \right)$ 
is negative for $1 < \d<2$. 
\end{proof} 

\begin{cor}
For $\d \ge 2$, the  Ces\`aro $(C,\d)$ means $\s_R^\d(f)$ define positive linear transformations on $L^1(\RR^2)$; 
the order of summability to assure positivity is best possible. 
\end{cor}

By Corollary \ref{cor:C=R}, this also proves Theorem 1.1. As in the case of $\ell_1$, the positivity of the kernel and its
proof are motivated by the corresponding result on the summability of the Fourier series. Indeed, it was proved in 
\cite{X10} that the Ces\`aro  $(C,\delta)$ means of the Fourier series associated with the hexagonal lattice are 
nonnegative if $\delta \ge 2$. 

\section{Positive definite hexagonal radial functions}
\setcounter{equation}{0}

In this section we consider hexagonal invariant functions that depend only on $\|\cdot\|_{\sH}$, which we call
$\|\cdot\|_{\sH}$ radial functions.  We start with the proof of Theorem \ref{thm2}, which we restate below, that 
gives a sufficient condition for a $\|\cdot\|_\sH$ radial function to be positive definite.

\begin{thm}\label{thm3}
Let $\phi \in C_b(\RR_+)$. The function $\phi(\|\tb\|_\sH)$, $\tb \in \RR_{\sH}^3$, is positive definite on $\RR_{\sH}^3$ 
if there exists an increasing bound function $\a$ on $\RR_+$ such that 
\begin{equation} \label{eq:m_H}
 \phi(t) = \int_0^\infty  m_2(t u)  d \a(u) \quad \hbox{with} \quad m_2(\xi) = \int_\xi^\infty \frac{\sin (u)}{u} du. 
\end{equation}
\end{thm}

\begin{proof}
By Bochner's theorem, $\phi(\|\tb\|_\sH)$ is positive definite exactly when it is the Fourier transform of a nonnegative,
integrable function, say $\Phi$, on $\RR_\sH^3$. By \eqref{eq:d-Drho}, we need to show that 
$$
  \Phi(\tb) = \int_0^\infty \phi(\rho)E_\rho(\tb) d\rho 
$$
is nonnegative. For $\rho$ given in \eqref{eq:m_H}, it is sufficient to write 
$$
  \Phi(\tb) = \int_0^\infty \left [ \int_0^\infty E_\rho(\tb) m_2(\rho u) d\rho \right] d \a(u)
$$ 
and prove that the inner integral is nonnegative. Using the fact that $E_{\rho/u} (\tb) = u^{-1} E_\rho(\tb/u)$, which
follows immediately from \eqref{eq:Erho}, it is enough to prove that 
$$
 J(\tb):= \int_0^\infty E_{\f32\rho}(\tb) m_2(\rho) d\rho \ge 0, \qquad \tb \in \RR_\sH^3. 
$$
Let $\chi_E$ be the characteristic function of the set $E \subset \RR_\sH^3$. We need the evaluation
\begin{align*}
 & \int_0^\infty \sin (u \rho) \int_\rho^\infty \frac{\sin s}{s} ds d\rho  = \int_0^\infty  \frac{\sin s}{s} \int_0^s \sin (u \rho) d\rho ds \\
 & \qquad \qquad = \frac{1}{u} \int_0^\infty \sin s (1- \cos (s u))\frac{ds}{s} = \frac{\pi}{4 u} (1-\sign (1-|u|) ).
\end{align*}
Together with the fact that $(t_1-t_2) + (t_2-t_3) +(t_3-t_1) =0$ and  $1+ \sign(1-|t_1-t_2|) = 2 \chi_{ \{|t_1-t_2|\le 1\}}(\tb)$, 
it follows from \eqref{eq:Erho} that 
$$
 J(\tb) = - \frac{3 \pi}{2} \left[\frac{\chi_{ \{|t_1-t_2|\le 1\}}(\tb)}{(t_2-t_3)(t_3-t_1)}
          +\frac{\chi_{ \{|t_2-t_3|\le 1\}}(\tb)}{(t_3-t_1)(t_1-t_2)}
          +\frac{\chi_{ \{|t_3-t_1|\le 1\}}(\tb)}{(t_1-t_2)(t_2-t_3)}\right],
$$
The value of $J(\tb)$ depends on regions of $\tb \in \RR_\sH^3$ determined by the support sets of the 
three characteristic functions. In the region $E_{---}:=\{\tb: |t_1-t_2| < 1, 
|t_2-t_3| < 1, |t_3-t_1| < 1\}$, $J(\tb) =0$ since $(t_1-t_2) + (t_2-t_3) +(t_3-t_1) =0$. We now consider
the region $E_{--+}:=\{\tb: |t_1-t_2| < 1, |t_2-t_3| < 1, |t_3-t_1| > 1\}$. In this case, 
$$
  J(\tb) =  - \frac{3 \pi}{4} \left[\frac{1}{(t_2-t_3)(t_3-t_1)}+\frac{1} {(t_3-t_1)(t_1-t_2)}\right]
         = \frac{3 \pi}{4} \frac{1}{(t_1-t_2)(t_2-t_3)},
$$
where the second equality follows from $t_1+t_2+t_3 =0$, which is positive if $t_1-t_2$ and $t_2-t_3$ have 
the same sign. Assume these two factors have different sign, say
$0 < t_1-t_2$ and $t_2-t_3 <0$. Then $\tb \in E_{--+}$ satisfies
$$
  0 < t_1-t_2 < 1, \quad -1 < t_2-t_3<0, \quad |t_3-t_1| > 1.
$$
The third inequality has two possibilities. In the case of $t_3-t_1 > 1$,  the second and the third inequalities imply 
that $t_1<t_2$, which contradicts the first inequality. In the case of $t_3-t_1< -1$, the first and the third inequality 
imply that $t_3 < t_2$, which contradicts the second inequality. Consequently, the set $E_{--+}$ does not 
contain elements for which $0 < t_1-t_2$ and $t_2-t_3 <0$, nor does it contain elements for which $ t_1-t_2<0$ 
and $0< t_2-t_3$. As a result, $J(\tb)$ is nonnegative for $\tb \in E_{--+}$. By symmetry, this holds for permutations
of $E_{--+}$. Next we consider the region $E_{-++}:=\{\tb: |t_1-t_2| < 1, |t_2-t_3| > 1, |t_3-t_1| > 1\}$, for which
$$
  J(\tb) = - \frac{3 \pi}{4}  \frac{1}{(t_2-t_3)(t_3-t_1)}
$$
is nonnegative if $t_2-t_3$ and $t_3-t_1$ have the opposite sign. Assume those two factors have the same sign,
say, $t_2-t_3>0$ and $t_3-t_1>0$, then $t_2-t_1 = t_2-t_3+t_3-t_1 >9$, so that $\tb \in E_{-++}$ satisfies $t_2-t_3>1$, 
$t_3-t_1 >1$ and $t_2-t_1 <1$. However, the first two inequalities imply that $t_2> t_1+2$, which contradicts the 
third inequality. Hence, the set $E_{-++}$ does not contain $\tb$ for which $t_2-t_3$ and $t_3-t_1$ have the same 
sign. Consequently, $J(\tb)$ is nonnegative on $E_{-++}$. By symmetry, this holds for permutations of $E_{-++}$. 
Finally, it is evident that $J(\tb) =0$ on $E_{+++}$, the definition of which should be obvious by now.  Thus, we have
proved that $J(\tb) \ge 0$ for all $\tb \in \RR_\sH^3$, which complete the proof of the theorem. 
\end{proof}

As shown in the proof, the support set of the function $J(\tb)$, $\tb \in \RR_\sH^3$, is relatively small. 
The graph of the function looks like a hexagonal spider (that has six legs), which is depicted in the Figure 3.

\begin{figure}[ht]
\centerline{
\includegraphics[width=0.75\textwidth]{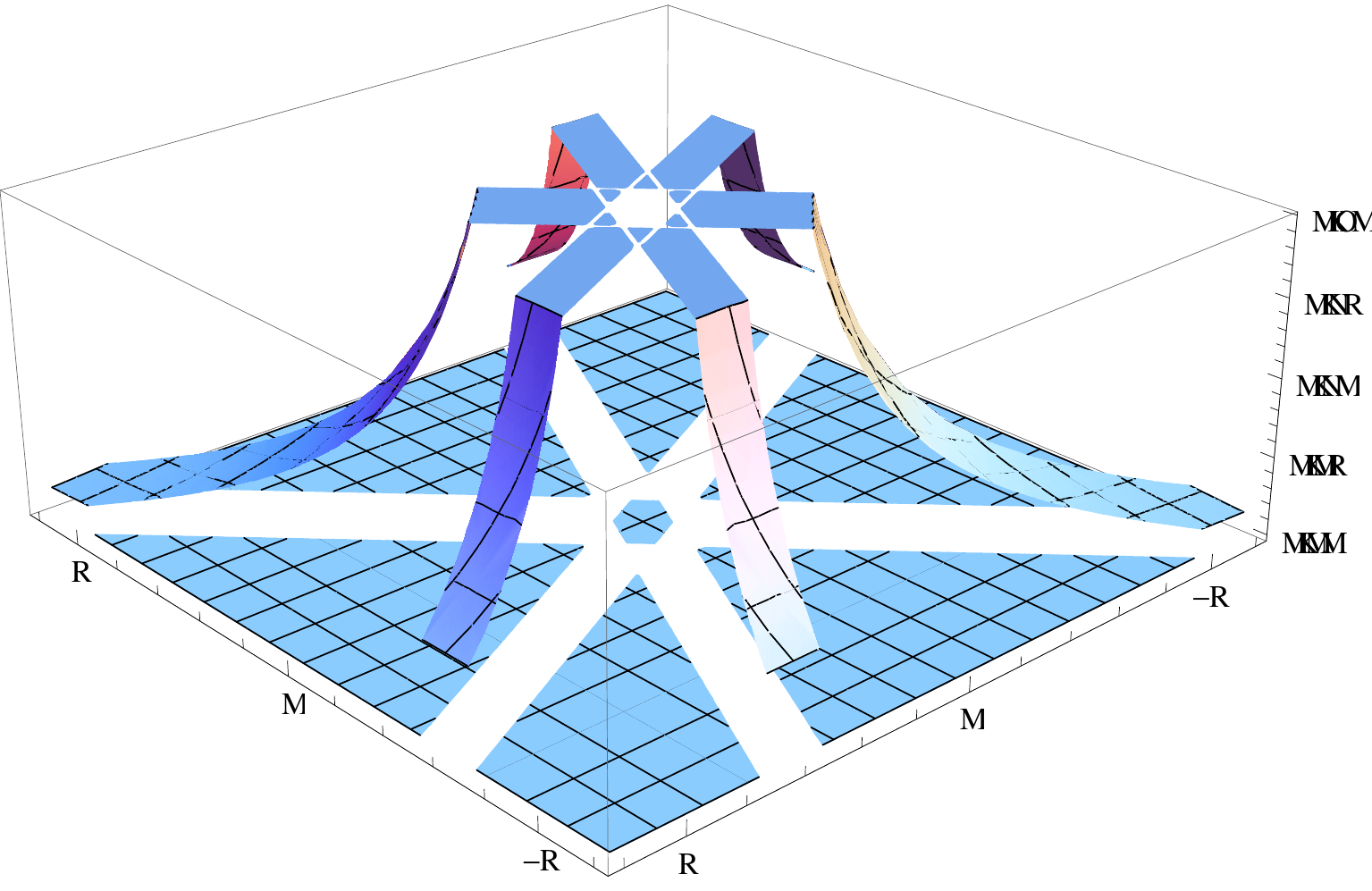}}
\caption{The function $J(\tb)$.}
\end{figure}

We do not know if the sufficient condition in the above theorem is necessary. The theorem shows that the 
class of positive definite $\|\cdot\|_\sH$ radial functions is at least as large as the class of positive
definite $\ell_1$ radial functions. In the latter case, the condition is necessary and it is established by
writing the corresponding kernel in terms of B-spline functions. An analogue representation can be given
in the hexagonal setting, which we discuss below. 

Recall that the first divided difference $[a,b]f$ is defined by 
$$
       [a,b] f := \frac{f(b) - f(a)}{b-a}, \qquad a, b\in \RR, \quad a \ne b, 
$$ 
which can also be written as 
$$
  [a,b] f = \int_{\RR} f'(u) B(u|a,b) du \quad\hbox{with}\quad  B(u|a,b): = \begin{cases}\frac{1}{b-a} & \hbox{if $b > u> a$} \\
  \frac{1}{a-b} & \hbox{if $a > u>b$} \\ 0 & \hbox{otherwise}\end{cases}.
$$
The function $B(\cdot| a,b)$ is the simplest example of $B$-spline functions and it evidently satisfies
$$
    B(u|a,b) \ge 0, \quad u \in \RR, \quad \hbox{and} \quad \int_\RR B(u|a,b) du =1.
$$

\begin{prop} \label{prop:E-kernel}
For $\tb \in \RR_{\sH}^3$ and $u \in \RR$, define 
\begin{align*}
  M_1(u|\tb) & :=   { B(u|t_1-t_3, t_2-t_3)+ B(u|t_2-t_1, t_3-t_1)+B(u|t_3-t_2, t_1-t_2)}.\\
  M(u|\tb) & := \frac12 \left[M_1(u|\tb)+M_1(u|-\tb) \right].
\end{align*}
Then, for $\rho > 0$, 
\begin{align}\label{eq:Erho2}
  E_\rho(\tb)  = 2 \rho \int_0^\infty \cos \left(\tfrac{2}{3} \rho u\right) M(u| \tb)du. 
\end{align}
\end{prop}

\begin{proof}
In the expression of $E_\rho$ in \eqref{eq:Erho}, we apply the partial fraction
$$
   - \frac{t_1-t_2}{(t_2-t_3)(t_3-t_1)} = \frac{1}{t_2-t_3}+  \frac{1}{t_3-t_1} 
$$
to the first term in the right hand side and two analogues partial fractions to the other two terms. Rearranging 
the terms leads to 
\begin{align*} 
  E_\rho(\tb)  & =  3\left( [t_1-t_3,t_2-t_3] \sin \left[\tfrac{2}{3} \rho\{\cdot\}\right] \right. \\
 &  \left. +[t_2-t_1,t_3-t_1] \sin \left[\tfrac{2}{3} \rho\{\cdot\}\right]   
          +[t_3-t_2,t_1-t_2] \sin  \left[\tfrac{2}{3} \rho\{\cdot\}\right] \right). \notag
\end{align*}
Writing the divided differences in terms of B-spline functions lead immediately to 
$$
  E_\rho(\tb) =  2 \rho \int_{-\infty}^\infty \cos \left(\tfrac{2}{3} \rho u\right) M_1(u| \tb)du.
$$
Directly from the definition, it is easy to verify that $B(u|a,b)$ satisfies $B(-u|a,b) = B(u|-a,-b)$, from which follows
$M_1(-u|\tb) = M_1(u|-\tb)$. Consequently, with our definition of $M(u|\tb)$, we can write the integral expression 
of $E_\rho(\tb)$ over $\tb \in \RR_\sH^3$ as the integral over $\RR_+$. 
\end{proof}

As a consequence of Propositions \ref{prop:FT2} and \ref{prop:E-kernel}, we immediate deduce the following result on
the inverse Fourier transform of hexagonal invariant functions. 

\begin{prop}
Let $\phi \in C_0(\RR_+)$ such that the function $u\mapsto u \phi(u)$ is in $L^1(\RR_+)$. Then 
\begin{equation}\label{eq:FTphi-H}
    \int_{\RR_\sH^3} \phi(\|\sb\|_\sH) e^{\frac{2i}{3} \sb\cdot \tb} d\sb = \int_0^\infty \psi (u) M(u|\tb) du
\end{equation}
where 
$$
   \psi(u) := 4 \int_0^\infty \rho \cos \left(\tfrac23\rho u\right)\phi(\rho) d\rho. 
$$
\end{prop}


It is easy to see that the function $M(u|\tb)$ satisfies 
$$
  M(u| \tb) \ge 0, \quad u \in \RR, \,\, \tb \in \RR_\sH^3, \quad\hbox{and} \quad \int_\RR M(u |\tb) du =3. 
$$
Furthermore, it also satisfies 
\begin{equation*}
  M(u| u \tb) = \frac{1}{u} M(1|\tb), \qquad u > 0. 
\end{equation*}
Thus, we only need to consider the $M(1|\tb)$. In Figure 4, we depict this function in the
regular rectangle coordinates. 
\begin{figure}[ht]
\centerline{
\includegraphics[width=1\textwidth]{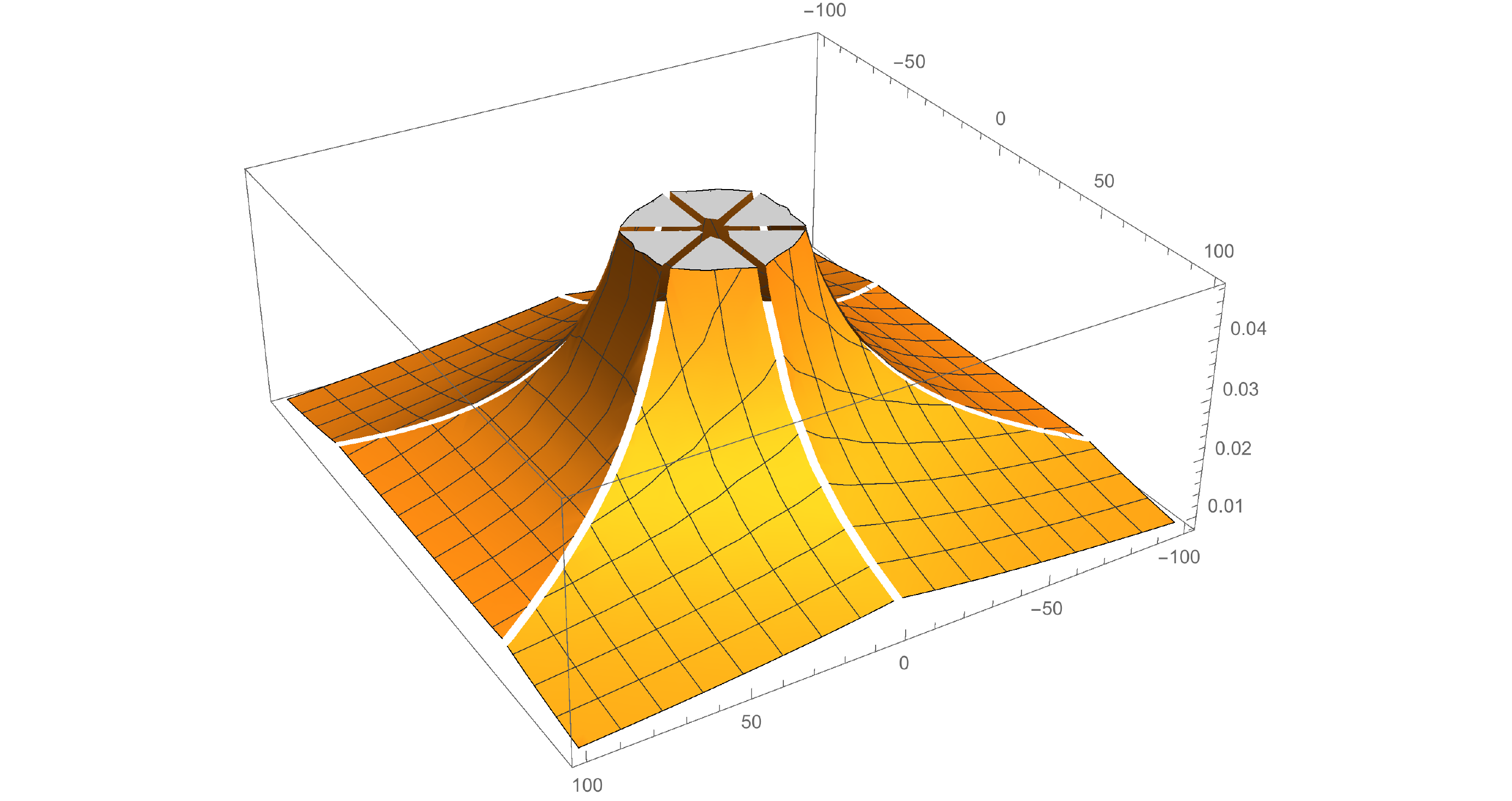}}
\caption{The function $M(1|\cdot)$ on the usual Cartesian coordinates.}
\end{figure}

The formula \eqref{eq:FTphi-H} suggests that we consider the Fourier transform of $M(u|\tb)$, as an analogue 
of the study in the $\ell_1$ case, for which the corresponding B-spline is $B(u|x_1^2,\ldots, x_d^2)$ for $x \in \RR^d$. 
However, the function $x \mapsto B(u|x_1^2,\ldots, x_d^2)$ is integrable on $\RR^d$, which warrants the existence 
of its Fourier transform, whereas the function $\tb \mapsto M(u|\tb)$ is not integrable on $\RR_\sH^3$. The latter
can be seen, for example, from the formula 
$$
   M_1(u|\tb) = \frac{1}{t_2-t_1}+ \frac{1}{t_2-t_3} = \frac{3 t_2}{(t_2-t_1)(t_2-t_3)}, \quad \tb \in  \Omega,
$$
where $\Omega = \{\tb \in \RR_\sH^3: t_2-t_3>1, t_2-t_1>1, t_3-t_1> -1, t_1-t_3>-1\}$.   

Let us also point out that, since $M(u|\tb) \ge 0$, if $\psi$ is nonnegative on $\RR_+$ then,  by \eqref{eq:FTphi-H},
the Fourier integral of $\phi(\|\cdot\|_\sH)$ is nonnegative. This is, however, not necessary. Indeed, if $\phi(u) 
= e^{-u}$, then $\Phi(\tb)= \int_{\RR_\sH^3} \phi(\sb) e^{\frac{2i}{3} \sb \cdot \tb} d \sb\ge 0$ by 
Example \ref{ex:2.3}. However, in this case
$$
  \psi(u) = 4 \int_0^\infty \rho \cos (u \rho) e^{-\rho}d\rho = \frac{4 (1-u^2)}{(1+u^2)^2},
$$ 
which is not nonnegative if $|u | > 1$. Hence, the expression \eqref{eq:FTphi-H} is far less useful than its 
counterpart in the $\ell_1$ case.

\end{document}